\documentclass[11pt]{article}
\usepackage{amssymb,amsmath,euscript}
\usepackage{url}

\newtheorem{proposition}{Proposition}[section]

\newtheorem{lemma}[proposition]{Lemma}
\newtheorem{theorem}[proposition]{Theorem}
\newtheorem{definition}[proposition]{Definition}

\def\F{{\EuScript F}}
\def\I{{\rm 1 \hskip-2.9truept l}}

\def\la{{\langle}}

\newcommand{\wh}{\widehat}
\newcommand{\wt}{\widetilde}
\def\R{{\mathbb R}}

\def\Q{{\mathbb Q}}
\def\a{\alpha}
\def\la{\lambda}
\def\La{\Lambda}

\def\de{\delta}

\def\Ga{\Gamma}
\def\ep{\varepsilon}

\def\epo{{\cal E}}

\def\kj{\overline{k}_j}
\def\O{\Omega}
\def\o{\omega}
\def\si{\sigma}
\def\th{\theta}
\def\vp{\varphi}
\def\z{\zeta}

\def\Re {{\rm Re}\,}
\def\Im {{\rm Im}\,}
\def\sI{\mathcal{I}}
\def\E{{\mathbb E}}
\def\P{{\mathbb P}}
\def\Z{{\mathbb Z}}
\def\N{{\mathbb N}}
\def\A{{\EuScript A}}
\def\K{{\EuScript K}_j (\wt{\rho}\,)}
\def\uK{\overline{{\EuScript K}}_j (\wt{\rho}\,)}
\def\M{{\EuScript M}}
\def\Ps {{\EuScript P}}
\def\ri{{\EuScript R}}

\def\supp{{\rm supp}}
\makeatletter \@addtoreset{equation}{section} \makeatother
\newcommand {\qed}%
{%
    {}\hfill
    {}\hfill
    {$\square $}%
    \vspace {0.3cm}%
    \pagebreak [2]%
    \par
}%
\newenvironment{proof}[1]{%
    \vspace{0.3cm}%
    \pagebreak [2]%
    \par%
    \noindent {\bf  Proof~#1\ }}{\qed}%
\newenvironment{remark}{%
    \vspace{0.3cm} \pagebreak [2]%
    \par%
    \refstepcounter{proposition}
    \noindent%
   {\bf Remark~\theproposition\  }}{\ }%
\def\Re {{\rm Re}\,}
\def\Im {{\rm Im}\,}
\def\E{{\mathbb E}}
\def\P{{\mathbb P}}
\def\Z{{\mathbb Z}}
\def\N{{\mathbb N}}
\def\M{{\EuScript M}}

\def\cqfd{$\square$}

\begin{document}

\title{An Optimal Uniform Modulus of Continuity for Harmonizable Fractional Stable Motion}

\author{Antoine Ayache\\
  Univ. Lille, \\
CNRS, UMR 8524 - Laboratoire Paul Painlev\'e,\\
F-59000 Lille, France\\  
E-mail: \texttt{antoine.ayache@univ-lille.fr}\\
\ 
\and
Yimin Xiao %\thanks{Research partially supported by NSF grant DMS-0706728.}
\\Department of Statistics and Probability
\\ Michigan State University  \\
East Lansing, MI 48824, U.S.A.\\
E-mail: \texttt{xiaoy@msu.edu}\\}

\maketitle

\begin{abstract}
Non-Gaussian Harmonizable Fractional Stable Motion (HFSM) is a natural and important extension of the well-known 
Fractional Brownian Motion to the framework of heavy-tailed stable distributions. It was introduced several decades ago; 
however its properties are far from being completely understood. In our present paper we determine the optimal power
of the logarithmic factor in a uniform modulus of continuity for HFSM, which solves an open old problem. The keystone 
of our strategy consists in Abel transforms of the LePage series expansions of the random coefficients of the wavelets 
series representation of HFSM. Our methodology can be extended to more general harmonizable stable processes and fields.
\end{abstract}

\medskip

      {Running head}: Optimal modulus of continuity of HFSM  \\

      {\it AMS Subject Classification (MSC2020 database)}: 60G52, 60G17, 60G22.\\


\section{Introduction and statements of the main results}
\label{sec:intro}
For any given constants $\a\in (0,2]$ and $H\in (0,1)$, the Harmonizable Fractional Stable Motion (HFSM) is 
the continuous real-valued 
symmetric $\a$-stable (S$\a$S) stochastic process $\{X(t), t\in\R\}$ defined as 
\begin{equation}
\label{eq:df-hfsm}
X(t):=\Re\bigg(\int_{\R}\frac{e^{it\xi}-1}{|\xi|^{H+1/\a}}\, d\wt{M}_\a(\xi)\bigg),
\end{equation} 
where $\widetilde{M}_\alpha$ is  a complex-valued rotationally invariant $\a$-stable random measure with 
Lebesgue control measure. When $\alpha = 2$,  $\{X(t), t\in\R\}$ is a Fractional Brownian Motion, usually 
denoted by $\{B_H (t), t\in\R\}$, whose sample path regularity and many other properties have been 
extensively studied in the literature. When $0 < \alpha < 2$,  $\{X(t), t\in\R\}$ is one of the most important 
non-Gaussian self-similar S$\a$S processes with stationary increments; we refer to the well-known book 
of Samorodnitsky and Taqqu \cite{ST94} for a systemic account on these processes and many topics 
related with stable distributions.
%(another such process is linear fractional stable motion whose regularity properties are significantly different from those of HFSM);
Non-Gaussian HFSM was introduced about 35 years ago by Cambanis and Maejima in \cite{CM89}; 
however, its properties are far from being completely understood. Generally speaking, study of this process 
has been difficult for various reasons: lack of ergodicity, lack of finite second moment, and so on. 
Sample path behavior of HFSM such as the uniform modulus of continuity on a compact interval has been 
studied by  K\^ono and Maejima \cite{KonoMaejima91}, Xiao \cite{Xiao10}, Boutard \cite{Boutard}, Ayache 
and Boutard \cite{AB17}, Panigrahi  et al \cite{PRX21}, by using the LePage series representation, modified 
chaining argument, and wavelet methods, respectively. 
% modified the existing chaining argument and made it amenable 
%to heavy-tailed random fields. This technique uses estimates of the lower order moments of the maximum increments over the two consecutive 
%steps of the chain to obtain a uniform modulus of continuity for stable and other heavy-tailed random fields.
%Panigrahi  et al \cite{PRX21}
More specifically, in their pioneering article \cite{KonoMaejima91},  K\^ono and Maejima applied a LePage 
series representation for the process $\{X(t), t\in\R\}$ to show that for any given numbers $H\in (0,1)$, $\a\in (0,2)$,  
$\wt{\rho} >0$, and arbitrarily small $\de>0$, almost surely
\begin{equation}
\label{eq:koma}
\sup_{-\wt{\rho}\le t'<t''\le \wt{\rho}}\,\frac{\big | X(t')-X(t'')\big |}{\big |t'-t'' \big |^{H}\big (\log(1+|t'-t''|^{-1})\big)^{1/\a+1/2+\de}}<+\infty.
\end{equation}
In his quite recent PhD thesis \cite{Boutard},  Boutard mainly established directional regularity and irregularity results 
on general (anisotropic) harmonizable stable random fields with stationary increments. In particular, Theorem 5.2.1 in 
\cite{Boutard} proves the following partial inverse to (\ref{eq:koma}): for any given numbers $H\in (0,1)$, $\a\in (0,2)$,  
$u<v$,  and arbitrarily small $\de>0$, almost surely
\begin{equation}
\label{thm:abou2:eq1}
\sup_{u\le t'<t''\le v}\,\frac{\big | X(t')-X(t'')\big |}{\big |t'-t'' \big |^{H}\big (\log(1+|t'-t''|^{-1})\big)^{1/\a-\de}}=+\infty.
\end{equation}
It is a natural question to wonder whether or not the power of the logarithmic factor in the uniform modulus of 
continuity (\ref{eq:koma}) is optimal. One is tempted to believe that it is not optimal, since in the Gaussian case 
$\a=2$ it is known for a long time that the optimal power for this factor is $1/2$ and not $1$. Yet, when $\a\in (0,2)$, 
the question has remained open for many years. In the case $\a\in (0,1)$, a 
negative answer to it has been given in \cite{AB17} and \cite{Boutard}. Indeed,   \cite[Theorem 3.8]{AB17}  or 
\cite[Theorem 4.1.2]{Boutard}  for general harmonizable stable random fields with stationary increments 
implies the following:

\begin{theorem} (\cite{AB17, Boutard})
\label{thm:abou1}
If $H\in (0,1)$ and $\a\in (0,1)$, then for any positive numbers $\wt{\rho}$ and $\de$, %be an arbitrarily small positive real-number. one has 
almost surely
\begin{equation}
\label{thm:abou1:eq1}
\sup_{-\wt{\rho}\le t'<t''\le \wt{\rho}}\,\frac{\big | X(t')-X(t'')\big |}{\big |t'-t'' \big |^{H}\big (\log(1+|t'-t''|^{-1})\big)^{1/\a+\de}}<+\infty.
\end{equation}
\end{theorem}

The proof of this result in \cite{AB17} and \cite{Boutard} %implying Theorem \ref{thm:abou1}, 
makes an essential use of a random wavelet series representation for harmonizable stable fields with stationary increments 
(see for instance Section 2 in \cite{AB17}). 
In the particular case of the HFSM $\{X(t), t\in\R\}$ this representation can be expressed, for all $H\in (0,1)$ and $\a\in (0,2]$, as
\begin{equation}
\label{eq:rep-hfsm}
X(t)=\sum_{(j,k)\in\Z^2} 2^{-jH}\Re\big (\ep_{\a,j,k}\big)\big (\Psi_{\a,H}(2^j t-k)-\Psi_{\a,H}(-k)\big).
\end{equation}
In the above, for every  $(j,k)\in\Z^2$, $\ep_{\a,j,k}$ is a complex-valued $\a$-stable random variable defined as
 \begin{equation}
 \label{def:ep}
\ep_{\a,j,k}:=\int_{\R} 2^{-j/\a}e^{ik2^{-j}\xi}\,\widehat{\psi}(-2^{-j}\xi)\,d\wt{M}_\a (\xi),
\end{equation}
and $\Psi_{\a,H}$ is the deterministic real-valued function in the Schwartz class $S(\R)$ defined as 
 \begin{equation}
 \label{eq:df-psiah}
 \Psi_{\a,H}(y):=\int_{\R}e^{iy\xi}\frac{\widehat{\psi}(\xi)}{|\xi|^{H+1/\alpha}}\;d\xi,\quad \mbox{for all $y\in\R$},
 \end{equation}
where, $\wh{\psi}$ is the Fourier transform of a real-valued univariate Meyer mother wavelet (see e.g.
  \cite{LM86,meyer92,Daubechies}). Recall that $\psi$ and $\wh{\psi}$ 
 belong to $S(\R)$ and that $\wh{\psi}$ is compactly supported with support satisfying
 \begin{equation}
 \label{eq:psihat}
 \supp\,\wh{\psi}\subseteq\ri_0:=\Big\{\xi\in\R\,:\,\frac{2\pi}{3}\le |\xi|\le \frac{8\pi}{3}\Big\}.
 \end{equation} 
Notice that the random series in (\ref{eq:rep-hfsm}) is almost surely absolutely convergent for each fixed 
$t\in\R$ (see Proposition 2.10 in  \cite{AB17}), and that it is almost surely uniformly convergent in $t$ on 
each compact interval (see (2.53) and Propositions 2.15, 2.16 and 2.17 in \cite{AB17}, see also Theorem 
2.7 in \cite{AShX20}). We mention in passing that \cite{AShX20} has extended the representation 
(\ref{eq:rep-hfsm}) to the Harmonizable Fractional Stable Sheet which is indexed by $\R^N$ and whose 
increments are non-stationary. 

Thanks to stochastic integral representation (\ref{def:ep}), the complex-valued $\a$-stable process 
$\big\{\ep_{\a,j,k},\, (j,k)\in\Z^2\big\}$ can be expressed as a LePage random series as shown by Proposition 
4.3 in \cite{AB17} which will be recalled at the beginning of Section \ref{sec:key} of the present article. 
This useful representation allowed \cite{AB17} and \cite{Boutard} to establish the almost sure fundamental 
inequality (see for instance (2.35) in \cite{AB17}):
\begin{equation}
\label{ineq1f:ep}
\big |\Re (\ep_{\a, j,k})\big |\le C_{\a,\de} (1+|j|)^{1/\a+\de},\quad\mbox{for all $\a\in (0,1)$, $\de>0$ and $(j,k)\in\Z^2$,}
\end{equation}
which is one of the main ingredients in the proof of Theorem \ref{thm:abou1} via the wavelet representation (\ref{eq:rep-hfsm}). 
For establishing the fundamental inequality (\ref{ineq1f:ep}), in which $C_{\a,\de}$ denotes a positive finite random 
variable only depending on $\a$ and $\de$, \cite{AB17} and \cite{Boutard} made an essential use of the fact that, 
when $\a\in (0,1)$, the LePage series representation of $\ep_{\a,j,k}$ can be bounded, uniformly in $k\in\Z$, 
by an almost surely convergent positive random series. Unfortunately, this method can hardly be extended to the case 
$\a\in [1,2)$ since the latter random series is no longer convergent. The divergence of this series creates a 
major difficulty.  Because of it, only a weaker form of the inequality (\ref{ineq1f:ep}) was obtained 
in \cite{AB17} and \cite{Boutard} (see for instance (2.36) in \cite{AB17}). Namely, if $\a\in [1,2)$, then for any $\de>0$,  
almost surely
\begin{equation}
\label{ineq2f:ep}
\big |\Re (\ep_{\a, j,k})\big |\le C_{\a,\de} (1+|j|)^{1/\a+\de}\sqrt{\log \big (3+|j|+|k|\big)} \quad\mbox{for all %$\a\in [1,2)$, $\de>0$ 
$(j,k)\in\Z^2$.}
\end{equation}
In contrast with  (\ref{ineq1f:ep}),   the inequality (\ref{ineq2f:ep}) is not sharp enough for improving the uniform 
modulus of continuity (\ref{eq:koma}) of HFSM via the wavelet representation (\ref{eq:rep-hfsm}).

In our present article we will introduce a new strategy for dealing with the major difficulty  described above. 
The starting point of this new strategy is to apply an Abel transform to the LePage series representation of 
$\ep_{\a,j,k}$. Thanks to the new strategy we are able to prove the following crucial proposition.

\begin{proposition}
\label{prop:fund}
Let $\a\in [1,2)$, $\de>0$ and $\rho>0$ be arbitrary and fixed. There exist an event $\O^*$ of probability $1$ 
and a positive finite random variable $C$, which depends on $\a$, $\de$ and $\rho$, such that on $\O^*$, 
the inequality
\begin{equation}
\label{prop:fund:eq2}
\big |\Re (\ep_{\a, j,k})\big |\le C (1+j)^{1/\a+\de}
\end{equation}
holds for all $(j,k)\in\Z_+\times\Z$ satisfying 
\begin{equation}
\label{prop:fund:eq1}
|2^{-j}k|\le\rho.
\end{equation}
\end{proposition} 

%Then the fundamental 
Proposition \ref{prop:fund} allows us to establish the main theorem of this paper, 
which shows that the improved uniform modulus of continuity for the HFSM $\{X(t), t\in\R\}$ 
provided by Theorem \ref{thm:abou1} when $\a\in (0,1)$ is also valid when $\a\in [1,2)$.

\begin{theorem}\label{thm:main1}
For any  $H\in (0,1)$, $\a\in [1,2)$, $\de>0$ and $\wt{\rho}>0$, the inequality (\ref{thm:abou1:eq1}) holds
almost surely.
\end{theorem}

For $H\in (0,1)$ and $\a\in (0,2)$, the uniform modulus of continuity in 
Theorems \ref{thm:abou1} and \ref{thm:main1} for the HFSM is more precise than the results in 
\cite{KonoMaejima91, Xiao10, PRX21}  on the uniform modulus of continuity for stable random 
fields including HFSM. We remark that the methods in the aforementioned references are different 
from the wavelet methods in \cite{AB17, Boutard} and in the present paper. Also, we point out that 
the methodology of the present article can be extended to harmonizable stable processes and 
fields which are more general than HFSM; this is what we intend to do in future works.

It follows from (\ref{thm:abou2:eq1}) that the uniform modulus of 
continuity for HFSM given by Theorems \ref{thm:abou1}  and \ref{thm:main1}  
 %when $(H,\a)\in (0,1)\times (1,2)$ 
is quasi-optimal. The following theorem improves (\ref{thm:abou2:eq1}) and shows that optimality of 
the uniform modulus of continuity for HFSM is fundamentally different from that for the Gaussian case 
$(\alpha = 2$) of the Fractional Brownian Motion $\{B_H (t), t\in\R\}$, where it is known that for any fixed 
real numbers $u<v$, almost surely,
\[
0<\sup_{u\le t'<t''\le v}\,\frac{\big | B_H(t')-B_H(t'')\big |}{\big |t'-t'' \big |^{H}\big (\log(1+|t'-t''|^{-1})\big)^{1/2}}<+\infty.
\]
%and
%\[
%\sup_{u\le t'<t''\le v}\,\frac{\big | B_H(t')-B_H(t'')\big |}{\big |t'-t'' \big |^{H}\big (\log(1+|t'-t''|^{-1})\big)^{1/2-\de}}=+\infty,\quad\mbox{for all $\de>0$.}
%\]

\begin{theorem}
\label{thm:main2}
For any given real numbers $H\in (0,1)$, $\a\in (0,2)$ and $u<v$, the equality 
(\ref{thm:abou2:eq1}) with  $\de=0$ holds almost surely:
\begin{equation}
\label{thm:abou2:eq1b}
\sup_{u\le t'<t''\le v}\,\frac{\big | X(t')-X(t'')\big |}{\big |t'-t'' \big |^{H}\big (\log(1+|t'-t''|^{-1})\big)^{1/\a}}=+\infty.
\end{equation}
\end{theorem}

%Theorems \ref{thm:abou1} and \ref{thm:main2} leave two problems that are still open. 
%\begin{itemize}
%\item[(i)]\, The case of $\alpha = 1$ is left untreated by Theorems 1.1 and 1.3. 
%%\begin{conjecture}
%%\label{rem:open-que}
%%It is still an open problem to obtain an optimal uniform modulus of continuity for the HFSM when $\a=1$. 
%We conjecture that %the optimal one provided by Theorems \ref{thm:abou1} and \ref{thm:main2} for any 
%%$\a\in (0,2)  \setminus\{1\}$
%(\ref{thm:abou1:eq1})  remains valid when $\a=1$.
%%except in the particular case $\a=1$. %We intend to solve it except in the particular case $\a=1$.
%%\end{conjecture}
%\item[(ii)]\, Given (\ref{thm:abou1:eq1}) and (\ref{thm:abou2:eq1b}), it would be interesting to ask if there is 
%a  function $L: \R_+ \to \R_+$ that is slowly varying at the origin such that 
%\begin{equation}
%\label{thm:abou2:eq1c}
%0< \sup_{u\le t'<t''\le v}\,\frac{\big | X(t')-X(t'')\big |} { |t'-t'' \big |^{H}  L(|t'-t''|)} <+\infty.
%\end{equation}
%%function the exact uniform modulus of 
%%continuity for the HFSM exists. 
%If such a function $L$ does not exist, is there a criterion on $L$ (e.g., an integral test) that can decide when 
%the supremum in (\ref{thm:abou2:eq1c}) is 0 or $\infty$?
%\end{itemize}

Theorems \ref{thm:abou1} and \ref{thm:main2} leave the following open questions. %problem that is still open. 
\begin{itemize}
\item[(i)]\, Is there a  function $L: \R_+ \to \R_+$ that is slowly varying at the origin such that 
\begin{equation}
\label{thm:abou2:eq1c}
0< \sup_{u\le t'<t''\le v}\,\frac{\big | X(t')-X(t'')\big |} { |t'-t'' \big |^{H}  L(|t'-t''|)} <+\infty?
\end{equation}
%function the exact uniform modulus of 
%continuity for the HFSM exists. 
\item[(ii)]\, If the answer to Question (i) is negative, is there a criterion on $L$ (e.g., an integral test) that can decide 
when the supremum in (\ref{thm:abou2:eq1c}) is 0 or $\infty$?
\end{itemize}

The rest of our article is organized as follows. Section \ref{sec:key} is devoted to the proof of the 
fundamental Proposition \ref{prop:fund}, and Section \ref{sec:p-main} to those of Theorems \ref{thm:main1} 
and \ref{thm:main2}.

\section{Proof of Proposition \ref{prop:fund}} 
\label{sec:key} 

First, we give a represntation of the complex-valued $\a$-stable stochastic process $\big\{\ep_{\a,j,k},\, (j,k)\in\Z^2\big\}$, 
defined through (\ref{def:ep}),  as a LePage-type random series. The following proposition is a reformulated version 
of Proposition 4.3 in \cite{AB17} in our setting.

\begin{proposition}
\label{prop:lepage}
Let $\a \in (0,2)$ be a constant and let the positive finite constant $a_{\a}:=\Big(\int_{0}^{+\infty}x^{-\a}\sin x\,dx\Big)^{-1/\a}$. Then
 \begin{equation}
 \label{lepage:ep}
\Big\{\ep_{\a,j,k}\,,\,\, (j,k)\in\Z^2\Big\}\stackrel{\mbox{{\tiny $d$}}}{=}\bigg \{a_{\a}\sum_{m=1}^{+\infty} 
\Gamma_{m}^{-1/\a}\varphi (\z_m)^{-1/\a}\,2^{-j/\a}e^{ik2^{-j}\z_m}\wh{\psi}(-2^{-j}\z_m) g_m\,,\,\, (j,k)\in\Z^2\bigg\},
\end{equation}
where $\stackrel{\mbox{{\tiny $d$}}}{=}$ means equality %the equality between these two stochastic processes holds in the sense 
of all finite-dimensional distributions. 
Notice that, for all $(j,k)\in\Z^2$, the random series in the right-hand side of (\ref{lepage:ep}) converges almost surely, and that 
$(\Gamma_m)_{m\in\N}$, $(\z_m)_{m\in\N}$,  and $(g_m)_{m\in\N}$ are three mutually independent sequences of 
random variables, defined on the same probability space, which satisfy the following three properties.
\begin{itemize}
\item[(i)] $(\Gamma_m)_{m\in\N}$ is a sequence of Poisson arrival times with unit rate; in other words there is a 
sequence $(\epo_n)_{n\in\N}$ of independent and identically distributed exponential random variables with parameter 
equals to $1$ such that 
\begin{equation}
\label{eq:ga}
\Ga_m=\sum_{n=1}^m \epo_n\, \ \hbox{ for all }\, m\in\N.
\end{equation}
\item[(ii)] $(\z_m)_{m\in\N}$ is a sequence of real-valued independent, identically distributed absolutely continuous random 
variables with probability density function $\vp$ given by $\vp(0)=0$ and 
\begin{equation}
\label{eq:vp}
\vp(\xi)=4^{-1}\ep\, |\xi|^{-1}\big (1+\big| \log |\xi|\big|\big)^{-1-\ep}\,\quad \mbox{for $\xi\in\R\setminus\{0\}$},
\end{equation} 
where $\ep$ is an arbitrarily small positive number; in fact $\ep$ will later be chosen in a convenient way 
related to $\de$ which appears in (\ref{prop:fund:eq2}).
\item[(iii)] $(g_m)_{m\in\N}$ is a sequence of complex-valued, independent, identically distributed, centered Gaussian 
random variables such that $\E\big (|\Re(g_m)|^\a\big)=1$.
\end{itemize}
\end{proposition}
 
% \begin{remark}
% \label{rem:as-pag}
 From now on, we will not distinguish the two stochastic processes in the left and right hand sides of the equality (\ref{lepage:ep}).
 % will always be identified.
% \end{remark}
\medskip 
 
 %Our present goal is to show the following fundamental proposition whose proof will be divided into %require 
 %several steps.

\begin{definition}
\label{def:riparts}
For each $m\in\N$, set
\begin{equation}
\label{def:riparts:eq0}
g_{0,m}:=\Re(g_m)\quad\mbox{and}\quad g_{1,m}:=\Im(g_m).
\end{equation}
Moreover, for any $(j,k)\in\Z_+\times\Z$, let
\begin{eqnarray}
\label{def:riparts:eq0bis}
&&\la_{0,m}^{j,k}:=\Re \Big(\varphi
(\z_m)^{-1/\a}\,2^{-j/\a} e^{ik2^{-j}\z_m}\wh{\psi}(-2^{-j}\z_m)\Big)\nonumber\\
&&\mbox{and}\\
&&\la_{1,m}^{j,k}:=\Im \Big(\varphi
(\z_m)^{-1/\a}\,2^{-j/\a} e^{ik2^{-j}\z_m}\wh{\psi}(-2^{-j}\z_m)\Big).\nonumber
\end{eqnarray}
Then, %one clearly has that
\begin{equation}
\label{def:riparts:eq1}
\Re \Big(\varphi
(\z_m)^{-1/\a}\,2^{-j/\a} e^{ik2^{-j}\z_m}\wh{\psi}(-2^{-j}\z_m)g_m\Big)=\la_{0,m}^{j,k}g_{0,m}-\la_{1,m}^{j,k}g_{1,m}\,,
\end{equation}
and consequently (see (\ref{lepage:ep}))  
\begin{equation}
\label{def:riparts:eq2}
\Re(\ep_{\a,j,k})=a_{\a}\sum_{m=1}^{+\infty} \Ga_m^{-1/\a}\big (\la_{0,m}^{j,k}g_{0,m}-\la_{1,m}^{j,k}g_{1,m}\big).
\end{equation}
Notice that the random series in (\ref{def:riparts:eq2}) is almost surely convergent, since the random series in (\ref{lepage:ep}) 
satisfies this property.
%where
%\begin{equation}
%\label{def:riparts:eq3}
%\chi_{j,k}^{l}:=\sum_{m=1}^{+\infty} \Ga_m^{-1/\a}\la_{l,m}^{j,k}g_{l,m}\,,\quad\mbox{for all $l\in\{0,1\}$.}
%\end{equation}
\end{definition}

\begin{definition}
\label{def:Smjk}
For any $l\in\{0,1\}$ and $(j,k)\in\Z_+\times\Z$, we set $S_{l,0}^{j,k}:=0$, and, for all $m\in\N$,
\begin{equation}
\label{def:Smjk:eq1}
S_{l,m}^{j,k}:=\sum_{n=1}^m \la_{l,n}^{j,k}g_{l,n}\,.
\end{equation}
\end{definition}

\begin{definition}
\label{def:ber-bin}
Let $\ri_0:=\big\{\xi\in\R\,:\, 2\pi/3\le |\xi|\le 8\pi/3\big\}$ be as in (\ref{eq:psihat}). For every $j\in\Z_+$,  
denote by $(\beta_n^j)_{n\in\N}$ the sequence of the independent and identically distributed 
Bernoulli random variables defined as 
\begin{equation}
\label{def:ber-bin:eq1}
\beta_n^j:=\I_{\ri_0} (2^{-j} \z_n)\,,
\end{equation}
and  by $(B_m^j)_{m\in\N}$ the sequence of the binomial random variables defined as
\begin{equation}
\label{def:ber-bin:eq2}
B_m^j:=\sum_{n=1}^m \beta_n^j\,.
\end{equation}
\end{definition}

\begin{lemma}
\label{lem:bou-lamjk}
Let $\mu_{\a,\ep}$ be the positive finite constant defined as
\begin{equation}
\label{lem:bou-lamjk:eq1}
\mu_{\a,\ep}:=\sup_{\xi\in\ri_0} \vp (\xi)^{-1/\a}|\wh{\psi}(\xi)|=(4/\epsilon)^{1/\a}\sup_{\xi\in\ri_0} |\xi|^{1/\a}(1+\log|\xi|)^{\frac{1+\epsilon}{\a}}|\wh{\psi}(\xi)|,
\end{equation}
where the last equality results from (\ref{eq:vp}) and (\ref{eq:psihat}). Then, almost surely  for all $(j,k)\in\Z_+\times\Z$, $l\in\{0,1\}$ and $n\in\N$,
\begin{equation}
\label{lem:bou-lamjk:eq2}
|\la_{l,n}^{j,k}|\le \mu_{\a,\ep}(1+j)^{\frac{1+\epsilon}{\a}}\beta_n^j\,.
\end{equation}
\end{lemma}

\begin{proof} \ It follows from (\ref{def:riparts:eq0bis}), (\ref{eq:vp}), the triangle inequality, (\ref{eq:psihat}), (\ref{def:ber-bin:eq1}) 
and (\ref{lem:bou-lamjk:eq1}) that almost surely, for all $(j,k)\in\Z_+\times\Z$, $l\in\{0,1\}$ and $n\in\N$, 
\begin{eqnarray*}
|\la_{l,n}^{j,k}| &\le &2^{-j/\a} \varphi (\z_n)^{-1/\a}\big|\widehat{\psi}(2^{-j}\z_n)\big|\\
&\le &  (4/\epsilon)^{1/\a}|2^{-j}\z_n|^{1/\a}\Big(1+j+\big|
\log|2^{-j}\z_n|\big|\Big)^{\frac{1+\epsilon}{\a}}|\widehat{\psi}(2^{-j}\z_n)|\\
&\le & (4/\epsilon)^{1/\a}|2^{-j}\z_n|^{1/\a}\Big(1+\big|
\log|2^{-j}\z_n|\big|\Big)^{\frac{1+\epsilon}{\a}}|\widehat{\psi}(2^{-j}\z_n)|(1+j)^{\frac{1+\epsilon}{\a}}\\
&\le & \mu_{\a,\ep}(1+j)^{\frac{1+\epsilon}{\a}}\beta_n^j\,,
\end{eqnarray*}
which shows that (\ref{lem:bou-lamjk:eq2})  holds.
\end{proof}

The following lemma provides the first upper bound for $|S_{l,m}^{j,k}|$.
\begin{lemma}
\label{lem:bou1-Smjk}
There exists a positive finite random variable $C'$ such that
\begin{equation}
\label{lem:bou1-Smjk:eq1}
|S_{l,m}^{j,k}|\le C' (1+j)^{\frac{1+\epsilon}{\a}}B_m^j \sqrt{\log(1+m)}\,
\end{equation}
almost surely for all $(j,k)\in\Z_+\times\Z$, $l\in\{0,1\}$ and $m\in\N$.
\end{lemma}  

\begin{proof}\ The lemma is a straightforward consequence of (\ref{def:Smjk:eq1}), the triangle inequality, 
Lemma \ref{lem:bou-lamjk}, Definition \ref{def:ber-bin}, (\ref{def:riparts:eq0}) and the following remark which can easily be derived from 
Lemma 1 in \cite{AT03}.
\end{proof} 

\begin{remark}
\label{rem:gau}
Let $(\wt{g}_m)_{m\in\N}$ be an arbitrary sequence of real-valued centered, identically distributed Gaussian random variables 
(which are not necessarily independent). Then, there is a positive finite random variable $C$ such that  almost surely,  
\begin{equation}
\label{rem:gau:eq1}
|\wt{g}_m|\le C \sqrt{\log(1+m)}\,, \quad\mbox{for all $m\in\N$.}
\end{equation}
\end{remark}

The following lemma provides the second upper bound for $|S_{l,m}^{j,k}|$.
\begin{lemma}
\label{lem:bou2-Smjk}
There exists a positive finite random variable $C''$ such that
\begin{eqnarray}
\label{lem:bou2-Smjk:eq1}
|S_{l,m}^{j,k}| &\le & C'' \sqrt{\bigg (\sum_{n=1}^m \Big |\la_{l,n}^{j,k}\Big |^2\bigg)\log \Big (3+j+|k|+m\Big)}\nonumber\\
&\le & C'' \mu_{\a,\ep}(1+j)^{\frac{1+\epsilon}{\a}} \sqrt{B_m^j\log\big(3+j+|k|+m\big)}
\end{eqnarray}
almost surely for all $l\in\{0,1\}$, $(j,k)\in\Z_+\times\Z$ and $m\in\N$.
\end{lemma} 

\begin{proof} \, First notice that the second inequality in (\ref{lem:bou2-Smjk:eq1}) is a straightforward 
consequence of the first inequality in it and of (\ref{lem:bou-lamjk:eq2}) and (\ref{def:ber-bin:eq2}). Hence, 
we only have to show that the first inequality in (\ref{lem:bou2-Smjk:eq1}) is satisfied.

For all $l\in\{0,1\}$, $(j,k)\in\Z_+\times\Z$ and $m\in\N$, let $\La_{l,m}^{j,k}$ be the event defined as 
\begin{equation}
\label{lem:bou2-Smjk:eq2}
\La_{l,m}^{j,k}:=\left\{\o\in\O,\, \big |S_{l,m}^{j,k}(\o)\big | >4\si_l\,\sqrt{\bigg (\sum_{n=1}^m 
\Big |\la_{l,n}^{j,k}(\o)\Big |^2\bigg)\log \Big (3+j+|k|+m\Big)}\,\right\},
\end{equation} 
where $\si_l>0$ denotes the common value of the standard deviations of the centered independent real-valued Gaussian random variables $g_{l,n}$, $n\in\N$. For proving the first inequality in (\ref{lem:bou2-Smjk:eq1}), it is enough to show that for $l\in\{0,1\}$, 
\begin{equation}
\label{lem:bou2-Smjk:eq3}
\sum_{j=0}^{+\infty}\sum_{k=-\infty}^{+\infty}\sum_{m=1}^{+\infty} \P\big (\La_{l,m}^{j,k}\big)<+\infty.
\end{equation}
Indeed, (\ref{lem:bou2-Smjk:eq3}) means that 
\[
\E\Big (\sum_{j=0}^{+\infty}\sum_{k=-\infty}^{+\infty}\sum_{m=1}^{+\infty} \I_{\La_{l,m}^{j,k}}\Big)<+\infty
\]
and consequently that 
\[
\sum_{j=0}^{+\infty}\sum_{k=-\infty}^{+\infty}\sum_{m=1}^{+\infty} \I_{\La_{l,m}^{j,k}}<+\infty,\quad\mbox{almost surely.}
\]
Thus, for $l\in\{0,1\}$, the random set of indices $\big\{(j,k,m)\in\Z_+\times\Z\times\N\,:\, \I_{\La_{l,m}^{j,k}}=1\big\}$ is 
almost surely finite. The latter fact implies that the positive random variable $C''$ defined as
\begin{equation}
\label{lem:bou2-Smjk:eq3bis}
C'':=\sup_{(l,j,k,m)\in\{0,1\}\times\Z_+\times\Z\times\N}\left\{4\si_l+\Bigg(\bigg (\si_l^2\sum_{n=1}^m \big |\la_{l,n}^{j,k}\big |^2\bigg)
\log \Big (3+j+|k|+m\Big)\Bigg)^{-1/2}\big |S_{l,m}^{j,k}\big |\I_{\La_{l,m}^{j,k}}\right\},
\end{equation}
with the conventions that $0^{-1/2}=+\infty$ and $(+\infty)0=0$, is almost surely finite. Moreover, it can easily be
 seen that the first inequality in (\ref{lem:bou2-Smjk:eq1}) holds when $C''$ is defined through (\ref{lem:bou2-Smjk:eq3bis}). 

Now we proceed with the proof of (\ref{lem:bou2-Smjk:eq3}).
Denote by $\E_{\z}(\cdot)$ the conditional expectation operator with respect to $\F_\z$, the $\si$-field spanned by 
the sequence of the random variables $(\z_m)_{m\in\N}$. It is clear that, for all $l\in\{0,1\}$, $(j,k)\in\Z_+\times\Z$ and $m\in\N$, 
\begin{equation}
\label{lem:bou2-Smjk:eq4}
\P\big (\La_{l,m}^{j,k}\big)=\E\Big(\I_{\La_{l,m}^{j,k}}\Big)=\E\Big (\E_{\z}\Big(\I_{\La_{l,m}^{j,k}}\Big)\Big).
\end{equation}
Then, by using the fact that the conditional distribution of $S_{l,m}^{j,k}$ with respect to $\F_\z$ is a centered Gaussian 
distribution with standard deviation equals
\[
\si_l\,\sqrt{\sum_{n=1}^m \Big |\la_{l,n}^{j,k}\Big |^2}
\]
and the fact that $4\,\sqrt{\log (3+j+|k|+m)}\ge 1$, we obtain  that almost surely
\[
\E_{\z}\Big(\I_{\La_{l,m}^{j,k}}\Big)\le \exp\bigg (-2^{-1}\Big(4\,\sqrt{\log (3+j+|k|+m)}\Big)^2\bigg)=\big (3+j+|k|+m\big)^{-8}.
\] 
Thus, it follows from (\ref{lem:bou2-Smjk:eq4}) that 
\[
\P\big (\La_{l,m}^{j,k}\big)\le (3+j+|k|+m\big)^{-8},
\]
which implies (\ref{lem:bou2-Smjk:eq3}).
\end{proof} 

\begin{remark}
\label{rem:ga}
It results from (\ref{eq:ga}) and the strong law of large number that $m^{-1} \Ga_m\xrightarrow[m\rightarrow+\infty]{a.s.} 1$, which  
entails that there exist two positive finite random variables $C'<C''$ such that almost surely, 
\begin{equation}
\label{rem:ga:eq1}
C' m\le \Ga_m\le C'' m\, \ \hbox{ for all $m\in\N$}.
\end{equation}
\end{remark}

\begin{remark}
\label{rem:epo}
Let $(\epo_n)_{n\in\N}$ be an arbitrary sequence of identically distributed exponential random variables (which are not 
necessarily independent). Then, there is a positive finite random variable $C$ such that  almost surely,  
\begin{equation}
\label{rem:epo:eq1}
\epo_n\le C \log(1+n)\, \ \mbox{for all $n\in\N$.}
\end{equation}
\end{remark}
To see this,  denote by $\la>0$ the common value of the parameters of the $\epo_n$'s. Then 
\[
\sum_{n=1}^{+\infty} \P\big (\epo_n>2\la^{-1}\log (n)\big)\le \sum_{n=1}^{+\infty} n^{-2} <+\infty.
\] 
Hence, (\ref{rem:epo:eq1}) follows from the Borel-Cantelli Lemma.

\begin{lemma}
\label{lem:abel}
For $l\in\{0,1\}$ and $(j,k)\in\Z_+\times\Z$, the random series 
\begin{equation}
\label{lem:abel:eq1}
\chi_{j,k}^{l}:=\sum_{m=1}^{+\infty} \big (\Ga_m^{-1/\a}-\Ga_{m+1}^{-1/\a}\big)S_{l,m}^{j,k}
\end{equation}
is almost surely absolutely convergent. Moreover, one has almost surely that
\begin{equation}
\label{lem:abel:eq2}
\Re(\ep_{\a,j,k})=a_\a\big(\chi_{j,k}^{0}-\chi_{j,k}^{1}\big).
\end{equation}
\end{lemma}

\begin{proof}\ In view of (\ref{lem:bou2-Smjk:eq1}),   the inequality 
\begin{equation}
\label{lem:abel:eq3}
\sqrt{B_m^j\log\big(3+j+|k|+m\big)}\le \Big (\sqrt{\log\big(3+j+|k|\big)}\,\Big)\sqrt{m\log(3+m)}\,,
\end{equation}
and the fact that $\a\in [1,2)$, in order to prove that the random series in (\ref{lem:abel:eq1}) is almost surely absolutely 
convergent, it is enough to show that almost surely 
\begin{equation}
\label{lem:abel:eq4}
\sup_{m\in \N} \,\frac{m^{1+\frac{1}{\a}}}{\log(1+m)}\, \big (\Ga_m^{-1/\a}-\Ga_{m+1}^{-1/\a}\big)<+\infty\,.
\end{equation}
From elementary calculations and (\ref{eq:ga}),  we derive that, for all $m\in\N$,
\begin{eqnarray}
\label{lem:abel:eq5}
\Ga_m^{-1/\a}-\Ga_{m+1}^{-1/\a}=\frac{\big (\Ga_m+\epo_{m+1}\big)^{1/\a}-\Ga_m^{1/\a}}{\Ga_m^{1/\a}\Ga_{m+1}^{1/\a}}
=\frac{1}{\Ga_{m+1}^{1/\a}}\bigg (\Big (1+\frac{\epo_{m+1}}{\Ga_m}\Big)^{1/\a}-1\bigg).
\end{eqnarray}
By using the inequality $(1+x)^{1/\a}-1\le \a^{-1} x$, for all $x\in\R_+$,  we derive that for all $m\in\N$, 
\begin{equation}
\label{lem:abel:eq6}
\Ga_m^{-1/\a}-\Ga_{m+1}^{-1/\a}\le\frac{\epo_{m+1}}{\a \Ga_m \Ga_{m+1}^{1/\a}}.
\end{equation}
Finally, combining (\ref{lem:abel:eq6}) with (\ref{rem:epo:eq1}) and the first inequality in (\ref{rem:ga:eq1}) yields
 (\ref{lem:abel:eq4}).

Now we show (\ref{lem:abel:eq2}). To this end we will use an Abel transform. For any integer $M\ge 2$, 
let $\Ps_M$ be the partial sum of order $M$ of the random series in (\ref{def:riparts:eq2}). That is 
\begin{equation}
\label{lem:abel:eq7}
\Ps_M:=a_{\a}\sum_{m=1}^{M} \Ga_m^{-1/\a}\big (\la_{0,m}^{j,k}g_{0,m}-\la_{1,m}^{j,k}g_{1,m}\big).
\end{equation}
By using the notations in Definition \ref{def:Smjk},  we can write $\Ps_M$ as 
\begin{eqnarray*}
&& \Ps_M = a_{\a}\bigg (\sum_{m=1}^{M} \Ga_m^{-1/\a}\Big (S_{0,m}^{j,k}-S_{0,m-1}^{j,k}\Big)-\sum_{m=1}^{M} \Ga_m^{-1/\a}
\Big (S_{1,m}^{j,k}-S_{1,m-1}^{j,k}\Big)\bigg)\nonumber\\
&& =a_{\a}\bigg (\sum_{m=1}^{M} \Ga_m^{-1/\a}S_{0,m}^{j,k}-\sum_{m=1}^{M-1} \Ga_{m+1}^{-1/\a}S_{0,m}^{j,k}-\sum_{m=1}^{M} 
\Ga_m^{-1/\a}S_{1,m}^{j,k}+\sum_{m=1}^{M-1} \Ga_{m+1}^{-1/\a}S_{1,m}^{j,k}\bigg)\nonumber\\
&&=a_{\a}\bigg (\Ga_M^{-1/\a}S_{0,M}^{j,k}-\Ga_M^{-1/\a}S_{1,M}^{j,k}+\sum_{m=1}^{M-1}\big (\Ga_m^{-1/\a}-\Ga_{m+1}^{-1/\a}\big)
S_{0,m}^{j,k}-\sum_{m=1}^{M-1}\big (\Ga_m^{-1/\a}-\Ga_{m+1}^{-1/\a}\big)S_{1,m}^{j,k}\bigg).
\end{eqnarray*} 
Thus, in view of (\ref{lem:abel:eq1}) and the fact that $\Ps_M$ converges almost surely to $\Re(\ep_{\a,j,k})$ as  $M \to +\infty$, 
we see that, for proving (\ref{lem:abel:eq2}), it is enough to show that, for $l\in\{0,1\}$, 
\begin{equation}
\label{lem:abel:eq8}
\Ga_M^{-1/\a}S_{l,M}^{j,k}\xrightarrow[M\rightarrow+\infty]{a.s.} 0\,.
\end{equation}
Putting together (\ref{lem:bou2-Smjk:eq1}), (\ref{lem:abel:eq3}), the first inequality in (\ref{rem:ga:eq1}) and the fact that $1/\a>1/2$, 
it follows that (\ref{lem:abel:eq8}) is satisfied. This finishes the proof.
\end{proof}

\begin{lemma}
\label{lem:sumBin}
For each $j\in\Z_+$, let
\begin{equation}
\label{lem:sumBin:eq1}
p_j:=\P\Big (\big\{\o\in\O,\,\, 2^{-j}\z_1\in \ri_0\big\}\Big)=\frac{\ep}{2}\int_{2^{j+1}\frac{\pi}{3}}^{2^{j+3}\frac{\pi}{3}}\frac{d\xi}{\xi (1+\log \xi )^{1+\ep}}\,,
\end{equation}
where the second equality follows from the facts that $\ri_0:=\big\{\xi\in\R\,:\, 2\pi/3\le |\xi|\le 8\pi/3\big\}$ and the probability density function of 
$\z_1$ is the even function $\vp$ given by (\ref{eq:vp}).
Then there is an event $\wt{\O}$  of probability $1$ with the following property: for each fixed $\eta \in (1/2,1)$, there exists a 
finite positive random variable $\wt{C}_\eta$ such that on $\wt{\O}$ %$\wt{C}_\eta$
\begin{equation}
\label{lem:sumBin:eq2}
B_m^j\le \wt{C}_\eta\big (p_j m+m^\eta\big),\quad\mbox{for all $(j,m)\in\Z_+\times\N$.}
\end{equation}
\end{lemma}

\begin{proof} \ For proving the lemma, it is enough to show that for any fixed $\eta\in(1/2,1)$,  
\begin{equation}
\label{lem:sumBin:eq3} 
\sum_{j=0}^{+\infty}\,\sum_{m=1}^{+\infty}\P\Big (\big |B_m^j-p_j m\big| >m^\eta\Big)<+\infty\,.
\end{equation}
Indeed, (\ref{lem:sumBin:eq3}) implies that 
\begin{equation}
\label{lem:sumBin:eq3ter}
\P\big (\wt{\O}_{\eta}\big)=1,
\end{equation}
where the event 
\begin{equation}
\label{lem:sumBin:eq3quat}
\wt{\O}_{\eta}:=\bigcup_{v=1}^{+\infty}\bigcap_{(j,m)\in\sI(v)}\big\{B_m^j\le p_j m+m^\eta\big\}
\end{equation}
with $\sI(v):=\big\{(j,m)\in\Z_+\times\N\,:\, j+m\ge v\big\}$ for each $v\in\N$. Thus, setting $\wt{\O}:=
\bigcap_{\eta\in \Q\cap (1/2,1)}\wt{\O}_{\eta}$, we can derive from (\ref{lem:sumBin:eq3quat}) and 
(\ref{lem:sumBin:eq3ter}) that,  for any $\eta\in (1/2,1)$,  the positive random variable  
\begin{equation}
\label{lem:sumBin:eq3bis} 
\wt{C}_\eta:=\sup_{(j,m)\in\Z_+\times\N}\Big\{\big ( p_j m+m^\eta\big)^{-1} B_m^j\Big\}
\end{equation}
is finite on the event $\wt{\O}$ of probability $1$. Moreover, it can easily be seen that  
(\ref{lem:sumBin:eq2}) holds on $\wt{\O}$ when $\wt{C}_\eta$ is defined through (\ref{lem:sumBin:eq3bis}).

Now it remains to prove (\ref{lem:sumBin:eq3}). For each $j\in\Z_+$ and $n\in\N$,  
denote by $\wt{\beta}_n^j$ the centered random variable defined as 
\begin{equation}
\label{lem:sumBin:eq4}
\wt{\beta}_n^j:=\beta_n^j -p_j\,,
\end{equation}
where $\beta_n^j$ is the Bernoulli random variable defined in (\ref{def:ber-bin:eq1}). Let $q$ be 
a positive integer which will be chosen later.  It follows from the Markov inequality  that for all $j\in\Z_+$ and $m\in\N$, 
\begin{equation}
\label{lem:sumBin:eq5}
\P\Big (\big |B_m^j-p_j m\big| >m^\eta\Big)\le m^{-2\eta q}\,\E\bigg ( \Big (\sum_{n=1}^m \wt{\beta}_n^j\Big)^{2q}\bigg)\,.
\end{equation}
In order to estimate $\displaystyle\E\bigg ( \Big (\sum_{n=1}^m \wt{\beta}_n^j\Big)^{2q}\bigg)$, we write it as
\begin{equation}
\label{lem:sumBin:eq6}
\E\bigg ( \Big (\sum_{n=1}^m \wt{\beta}_n^j\Big)^{2q}\bigg)=\sum_{1\le n_1,\ldots, n_{2q}\le m}\,\E\bigg (\prod_{p=1}^{2q} \wt{\beta}^j_{n_p}\bigg)\,.
\end{equation}
Notice that each $\displaystyle\prod_{p=1}^{2q} \wt{\beta}^j_{n_p}$ can be expressed, for some $r\in\{1,\ldots,2q\}$ and 
some distinct integers $\nu_1,\ldots,\nu_r$ satisfying to $1\le\nu_1<\ldots<\nu_r\le m$, as 
\begin{equation}
\label{lem:sumBin:eq7}
\prod_{p=1}^{2q} \wt{\beta}^j_{n_p}=\prod_{u=1}^{r} \Big (\wt{\beta}^j_{\nu_u}\Big)^{a_u}\,, 
\end{equation}
where $a_1,\ldots, a_r$ belong to $\{1,\ldots,2q\}$ and satisfy 
\begin{equation}
\label{lem:sumBin:eq8}
\sum_{u=1}^r a_u=2q\,.
\end{equation}
Then, by the independence of the centered random variables $\wt{\beta}^j_{\nu_1},\ldots , \wt{\beta}^j_{\nu_r}$,  we have
\begin{equation}
\label{lem:sumBin:eq9}
\E\bigg (\prod_{p=1}^{2q} \wt{\beta}^j_{n_p}\bigg)=\prod_{u=1}^{r} \E\bigg(\Big (\wt{\beta}^j_{\nu_u}\Big)^{a_u}\bigg)\,,
\end{equation}
which implies that the latter expectation vanishes as soon as $\min\{a_1,\ldots, a_r\}=1$. Thus, we only need 
to consider the case where $\min\{a_1,\ldots, a_r\}\ge 2 $, which implies $r\le q$ because of the equality 
(\ref{lem:sumBin:eq8}). Next notice that for any given $r\in\{1,\ldots, q\}$, distinct integers $1\le\nu_1<\ldots<\nu_r\le m$, 
and arbitrary numbers $a_1,\ldots, a_r$ belonging to $\{2,\ldots, 2q\}$ and satisfying (\ref{lem:sumBin:eq8}), there are exactly 
\[
\binom{2q}{a_1,\ldots, a_r}:=\frac{(2q)!}{a_1!\times\ldots\times a_r!}
\]
tuples $(n_1,\ldots, n_{2q})$ of numbers belonging to $\{1,\ldots,m\}$ for which the equality (\ref{lem:sumBin:eq7}) holds.
Thus, one can derive from (\ref{lem:sumBin:eq6}) and (\ref{lem:sumBin:eq9}) that 
\begin{equation}
\label{lem:sumBin:eq10}
\E\bigg ( \Big (\sum_{n=1}^m \wt{\beta}_n^j\Big)^{2q}\bigg)=\sum_{r=1}^{q}\,\,\sum_{1\le\nu_1<\ldots<\nu_r\le m}\,\,
\sum_{(a_1,\ldots, a_r)\in\A_{2q,r}}\binom{2q}{a_1,\ldots, a_r}\prod_{u=1}^{r} \E\bigg(\Big (\wt{\beta}^j_{\nu_u}\Big)^{a_u}\bigg)\,,
\end{equation}
where, for all $r\in\{1,\ldots, q\}$,
\begin{equation}
\label{lem:sumBin:eq11}
\A_{2q,r}:=\Big\{(a_1,\ldots, a_r)\in\{2,\ldots,2q\}^r,\,\,\,\,\sum_{u=1}^r a_u=2q\Big\}\,.
\end{equation}
Next, we claim 
\begin{equation}
\label{lem:sumBin:eq12}
\bigg|\E\bigg(\Big (\wt{\beta}^j_{n}\Big)^{a}\bigg)\bigg |\le p_j\,,\quad\mbox{for all $(n,j,a)\in\N\times\Z_+\times \{2,\ldots,2q\}$.}
\end{equation}
Indeed, it follows from (\ref{lem:sumBin:eq4}), the facts that $\beta_{n}^j$ is a Bernoulli random variable with parameter equals to 
$p_j$ (defined in (\ref{lem:sumBin:eq1})) and $a\ge 2$ that 
\[
\begin{split}
\bigg|\E\bigg(\Big (\wt{\beta}^j_{n}\Big)^{a}\bigg)\bigg| &=\big |p_j (1-p_j)^a+(1-p_j)(-p_j)^a\big|\\
&\le (1-p_j) \Big ((1-p_j)^{a-1}+p_j^{a-1}\Big) p_j \le p_j\,.
\end{split}
\]
This verifies (\ref{lem:sumBin:eq12}). Let $c_1(q)$ be the finite deterministic constant, only depending on $q$, defined as
\begin{equation}
\label{lem:sumBin:eq13}
c_1 (q):=\sum_{r=1}^{q}\,\,\sum_{(a_1,\ldots, a_r)\in\A_{2q,r}}\binom{2q}{a_1,\ldots, a_r}\,.
\end{equation}
Then, one can derive from (\ref{lem:sumBin:eq10}), (\ref{lem:sumBin:eq12}) and (\ref{lem:sumBin:eq13}) 
that for all $(j,m)\in\Z_+\times\N$,
\begin{equation}
\label{lem:sumBin:eq14}
\E\bigg ( \Big (\sum_{n=1}^m \wt{\beta}_n^j\Big)^{2q}\bigg)\le c_1 (q) m^q\, p_j.
\end{equation}
By combining (\ref{lem:sumBin:eq5}) and (\ref{lem:sumBin:eq14}), we obtain that for all $(j,m)\in\Z_+\times\N$,  
\begin{equation}
\label{lem:sumBin:eq15}
\P\Big (\big |B_m^j-p_j m\big| >m^\eta\Big)\le c_1 (q) m^{-(2\eta -1)q}\,p_j\,.
\end{equation}
Since $2\eta-1>0$, we choose the integer $q$ large enough so that $(2\eta-1)q>1$. Then, 
 (\ref{lem:sumBin:eq3})  follows from   (\ref{lem:sumBin:eq15}) and (\ref{lem:sumBin:eq1}).
\end{proof} 

We are now ready to prove Proposition \ref{prop:fund}.

\noindent{\bf Proof of Proposition \ref{prop:fund}}\, First, it follows from (\ref{lem:abel:eq4}) that there 
is a positive finite random variable $C_1$ such that almost surely, 
\begin{equation}
\label{prop:fund:eq2bis}
0<\Ga_m^{-1/\a}-\Ga_{m+1}^{-1/\a}\le C_1 m^{-1-1/\a}\log (1+m)\,, \quad\mbox{for all $m\in\N$.}
\end{equation}
Since $1/\a>1/2$, we can choose and fix  a constant $\eta_0\in (1/2,1)$ such that 
\begin{equation}
\label{prop:fund:eq3}
1+\frac{1}{\a}-\eta_0>1\,.
\end{equation}
For each $j\in\Z_+$, denote by $\M_j$ and $\overline{\M}_j$ the two nonempty sets of indices $m$ defined as 
\begin{equation}
\label{prop:fund:eq4}
\M_j:=\big\{m\in\N,\,\, p_j m\ge m^{\eta_0}\big\}
\end{equation}
and 
\begin{equation}
\label{prop:fund:eq5}
\overline{\M}_j:=\N\setminus \M_j=\big\{m\in\N,\,\, p_j m < m^{\eta_0}\big\},
\end{equation}
where the probability $p_j\in (0,1)$ is defined through (\ref{lem:sumBin:eq1}). Then, for every $j\in\Z_+$,
\begin{equation}
\label{prop:fund:eq5bis}
\N=\M_j \cup \overline{\M}_j\,, \quad\mbox{(disjoint union)}
\end{equation}
\begin{equation}
\label{prop:fund:eq6}
p_j m+m^{\eta_0}\le 2 p_j m\,,\quad\mbox{for all $m\in\M_j$}
\end{equation}
and 
\begin{equation}
\label{prop:fund:eq7}
p_j m+m^{\eta_0}< 2 m^{\eta_0}\,,\quad\mbox{for all $m\in\overline{\M}_j$.}
\end{equation}
In all the sequel $l\in\{0,1\}$ is arbitrary, and $j\in\Z_+$ and $k\in\Z$ are arbitrary and such that (\ref{prop:fund:eq1}) 
holds. It follows from (\ref{prop:fund:eq2bis}), (\ref{lem:bou1-Smjk:eq1}), (\ref{lem:sumBin:eq2}), (\ref{prop:fund:eq7}) and 
(\ref{prop:fund:eq3}) that almost surely, 
\begin{eqnarray}
\label{prop:fund:eq8}
&& \sum_{m\in\overline{\M}_j} \big (\Ga_m^{-1/\a}-\Ga_{m+1}^{-1/\a}\big)\big |S_{l,m}^{j,k}\big |\nonumber\\
&& \le C_2(1+j)^{\frac{1+\epsilon}{\a}}\sum_{m\in\overline{\M}_j} B_m^j \, m^{-1-1/\a}\big(\log (1+m)\big)^{3/2}\nonumber\\
&& \le C_3 (1+j)^{\frac{1+\epsilon}{\a}}\sum_{m\in\overline{\M}_j} m^{-(1+1/\a-\eta_0)}\big(\log (1+m)\big)^{3/2}\nonumber\\
&& \le C_4 (1+j)^{\frac{1+\epsilon}{\a}}\,,
\end{eqnarray}
where $C_2$ and $C_3$ are two positive finite random variables not depending on $j$ and $k$, and 
\[
C_4:=C_3 \sum_{m=1}^{+\infty} m^{-(1+1/\a-\eta_0)}\big(\log (1+m)\big)^{3/2}<+\infty\,.
\]
On the other hand, by using (\ref{prop:fund:eq2bis}), (\ref{lem:bou2-Smjk:eq1}), the inequality 
\[
\log \big (3+j+|k|+m\big)\le \log \big (3+j+|k|\big) \log (3+m),\quad\mbox{for all $(j,k,m)\in\Z_+\times\Z\times\N$,}
\]
 (\ref{prop:fund:eq1}), (\ref{lem:sumBin:eq2}), (\ref{prop:fund:eq6}), and the inequality $1/2+1/\a>1$, we have almost surely,
\begin{eqnarray}
\label{prop:fund:eq9}
&& \sum_{m\in\M_j} \big (\Ga_m^{-1/\a}-\Ga_{m+1}^{-1/\a}\big)\big |S_{l,m}^{j,k}\big |\nonumber\\
&& \le C_5 (1+j)^{\frac{1+\epsilon}{\a}}\, \sqrt{\log\big(3+j+|k|\big)}\sum_{m\in\M_j}\big (B_m^j\big)^{1/2} m^{-(1+1/\a)}\big (\log(3+m)\big)^{3/2}\nonumber\\
&& \le C_6 (1+j)^{\frac{1+\epsilon}{\a}} \, \sqrt{p_j (1+j)}\sum_{m\in\M_j}m^{-(1/2+1/\a)}\big (\log(3+m)\big)^{3/2}\nonumber\\
&& \le  C_7 (1+j)^{\frac{1+\epsilon}{\a}} \, \sqrt{p_j (1+j)}\,,
\end{eqnarray}
where $C_5$ and $C_6$ are two positive finite random variables not depending on $j$ and $k$, and 
\[
C_7:=C_6 \sum_{m=1}^{+\infty} m^{-(1/2+1/\a)}\big(\log (3+m)\big)^{3/2}<+\infty\,.
\] 
It follows from (\ref{lem:sumBin:eq1}) that
\[
\begin{split}
p_j &\le \frac{\ep\pi}{6} \big (2^{j+3}-2^{j+1}\big)\Big (\frac{2^{j+1}\pi}{3}\Big)^{-1}\bigg (1+\log \Big (\frac{2^{j+1}\pi}{3}\Big)\bigg)^{-1-\ep}\\
&\le 6 (1+j)^{-1-\ep}\,.
\end{split}
\]
Thus,
\begin{equation}
\label{prop:fund:eq10}
\sqrt{p_j (1+j)}\le \sqrt{6}\,,\quad\mbox{for all $j\in\Z_+$.}
\end{equation}
By (\ref{prop:fund:eq9}) and (\ref{prop:fund:eq10}), we have 
\begin{equation}
\label{prop:fund:eq11}
\sum_{m\in\M_j} \big (\Ga_m^{-1/\a}-\Ga_{m+1}^{-1/\a}\big)\big |S_{l,m}^{j,k}\big | \le  \sqrt{6}\, C_7 (1+j)^{\frac{1+\epsilon}{\a}}\,.
\end{equation}
Finally, by combining  (\ref{lem:abel:eq2}), the triangle inequality, (\ref{lem:abel:eq1}), (\ref{prop:fund:eq5bis}),
(\ref{prop:fund:eq8}), and (\ref{prop:fund:eq11}), we obtain almost surely  
\begin{eqnarray*}
&& \big |\Re(\ep_{\a,j,k})\big |\le a_\a \sum_{l=0}^1 \big |\chi_{j,k}^{l}\big|\le a_\a \sum_{l=0}^1\sum_{m\in\N} \big (\Ga_m^{-1/\a}-
\Ga_{m+1}^{-1/\a}\big)\big |S_{l,m}^{j,k}\big |\\
&& \le a_\a \sum_{l=0}^1\Big (\sum_{m\in\M_j} \big (\Ga_m^{-1/\a}-\Ga_{m+1}^{-1/\a}\big)\big |S_{l,m}^{j,k}\big |
+\sum_{m\in\overline{\M}_j} \big (\Ga_m^{-1/\a}-\Ga_{m+1}^{-1/\a}\big)\big |S_{l,m}^{j,k}\big |\Big)\\
&& \le 2 a_\a \big (C_4+\sqrt{6}\, C_7\big) (1+j)^{\frac{1+\epsilon}{\a}}.
\end{eqnarray*}
This proves (\ref{prop:fund:eq2}).
\cqfd

\section{Proofs of Theorems \ref{thm:main1} and \ref{thm:main2}}
\label{sec:p-main}

\noindent {\bf Proof of Theorem \ref{thm:main1}} \ Let $\wt{\rho}$ be a positive constant. % as in the statement 
%of the theorem. 
For every $j\in\Z_+$, the two nonempty sets $\K$ and $\uK$, which forms a partition of $\Z$, are defined as 
\begin{equation}
\label{thm:main1:eq0}
\K:=\big\{k\in\Z,\, |k|\le 2^j (\wt{\rho}+1)\big\}
\end{equation}
and 
\begin{equation}
\label{thm:main1:eq0bis}
\uK:=\big\{k\in\Z,\, |k|> 2^j (\wt{\rho}+1)\big\}.
\end{equation}
It follows from (\ref{eq:rep-hfsm}) that the HFSM $\{X(t), t\in\R\}$ can be expressed, for all $t\in\R$, as 
\begin{equation}
\label{thm:main1:eq1}
X(t)=X^{-}(t)+X_1^{+}(t)+X_2^{+}(t),
\end{equation}
where the process $\{X^{-}(t), t\in\R\}$ is the low-frequency part of the HFSM defined, for every $t\in\R$, as
\begin{equation}
\label{thm:main1:eq2}
X^{-}(t):=\sum_{j=-\infty}^{-1}\sum_{k=-\infty}^{+\infty} 2^{-jH}\Re\big (\ep_{\a,j,k}\big)\big (\Psi_{\a,H}(2^j t-k)-\Psi_{\a,H}(-k)\big),
\end{equation}
while the two processes $\{X_1^{+}(t), t\in\R\}$ and $\{X_2^{+}(t), t\in\R\}$, whose sum gives the high-frequency 
part of the HFSM, are defined, for each $t\in\R$, as
\begin{equation}
\label{thm:main1:eq3}
X_1^{+}(t):=\sum_{j=0}^{+\infty}\sum_{k\in\K} 2^{-jH}\Re\big (\ep_{\a,j,k}\big)\big (\Psi_{\a,H}(2^j t-k)-\Psi_{\a,H}(-k)\big)
\end{equation}
and
\begin{equation}
\label{thm:main1:eq4}
X_2^{+}(t):=\sum_{j=0}^{+\infty}\sum_{k\in\uK} 2^{-jH}\Re\big (\ep_{\a,j,k}\big)\big (\Psi_{\a,H}(2^j t-k)-\Psi_{\a,H}(-k)\big).
\end{equation}
It is known from Proposition 2.15 in \cite{AB17} that $\{X^{-}(t), t\in\R\}$ has almost surely infinitely differentiable paths. 
Thus, in view of (\ref{thm:main1:eq1}), for proving the theorem it is enough to show that, for all $H\in (0,1)$, $\a\in [1,2)$ 
and arbitrarily small $\de>0$, we have  almost surely
\begin{equation}
\label{thm:main1:eq5}
\sup_{-\wt{\rho}\le t'<t''\le \wt{\rho}}\,\frac{\big | X_1^+(t')-X_1^+(t'')\big |}{\big |t'-t'' \big |^{H}\big (\log(1+|t'-t''|^{-1})\big)^{1/\a+\de}}<+\infty
\end{equation}
and
\begin{equation}
\label{thm:main1:eq6}
\sup_{-\wt{\rho}\le t'<t''\le \wt{\rho}}\,\frac{\big | X_2^+(t')-X_2^+(t'')\big |}{\big |t'-t'' \big |^{H}\big (\log(1+|t'-t''|^{-1})\big)^{1/\a+\de}}<+\infty.
\end{equation}
First, we prove (\ref{thm:main1:eq5}). To this end, we apply Proposition \ref{prop:fund} with $\rho=\wt{\rho}+1$. Let $\O^*$ 
be the event of probability $1$  in this proposition, and let $t'$ and $t''$ be two arbitrary real numbers 
such that $-\wt{\rho}\le t'<t''\le \wt{\rho}$. It follows from  (\ref{thm:main1:eq0}) that  (\ref{prop:fund:eq1}) holds for all $(j,k)\in\Z_+\times\Z$ such that 
$k \in \K$. Hence, it results from (\ref{thm:main1:eq3}) and (\ref{prop:fund:eq2}) that on $\O^*$,
\begin{eqnarray}
\label{thm:main1:eq7}
&& \big | X_1^+(t')-X_1^+(t'')\big | \le \sum_{j=0}^{+\infty}\sum_{k\in\K} 2^{-jH}\big |\Re\big (\ep_{\a,j,k}\big)\big |\big |
\Psi_{\a,H}(2^j t'-k)-\Psi_{\a,H}(2^j t''-k)\big |\nonumber\\
&&\le  C_1 \sum_{j=0}^{+\infty}2^{-jH}(1+j)^{1/\a+\de} \sum_{k\in\K} \big |\Psi_{\a,H}(2^j t'-k)-\Psi_{\a,H}(2^j t''-k)\big |\nonumber\\
&&\le  C_1 \sum_{j=0}^{+\infty}2^{-jH}(1+j)^{1/\a+\de} \sum_{k\in\Z} \big |\Psi_{\a,H}(2^j t'-k)-\Psi_{\a,H}(2^j t''-k)\big |,
\end{eqnarray}
where $C_1$ is a positive finite random variable not depending on $t'$ and $t''$. Since the function 
$\Psi_{\a,H}$ belongs the Schwartz class, this function and its derivative $\Psi_{\a,H}'$ satisfy, for some finite constant 
$c_2$ and for all $y\in\R$,
\begin{equation}
\label{thm:main1:eq8}
\big |\Psi_{\a,H}(y)\big |+\big |\Psi_{\a,H}'(y)\big |\le c_2 \big (1+2\wt{\rho}+|y|\big)^{-3}.
\end{equation}
Observe that 
\begin{equation}
\label{thm:main1:eq8bis}
c_3:=\sup_{y\in\R}\sum_{k\in\Z} \big (1+|y-k|\big)^{-3}<+\infty.
\end{equation}
Next, let $j_0$ be the unique nonnegative integer such that 
\begin{equation}
\label{thm:main1:eq9}
2^{-j_0-1}(2\wt{\rho})<|t'-t''|\le 2^{-j_0} (2\wt{\rho}),
\end{equation}
that is 
\begin{equation}
\label{thm:main1:eq10}
j_0:=\left [\frac{\log\big ((2\wt{\rho})|t'-t''|^{-1}\big) }{\log (2)}\right],
\end{equation}
where $[\cdot]$ denotes the integer part function. Using the mean-value Theorem, (\ref{thm:main1:eq8}), (\ref{thm:main1:eq9}) 
and (\ref{thm:main1:eq8bis}), it can be shown, for all $j\in\{0,\ldots,j_0\}$, that 
\begin{eqnarray}
\label{thm:main1:eq11}
\sum_{k\in\Z} \big |\Psi_{\a,H}(2^j t'-k)-\Psi_{\a,H}(2^j t''-k)\big | &\le & c_2 2^j |t'-t''|\sum_{k\in\Z} \big (1+|2^j t'-k|\big)^{-3}\nonumber\\
&\le & c_2 c_3 2^j |t'-t''|.
\end{eqnarray}
It follows from (\ref{thm:main1:eq11}), (\ref{thm:main1:eq10}) and (\ref{thm:main1:eq9}) that
\begin{equation}
\label{thm:main1:eq12}
\begin{split}
& \sum_{j=0}^{j_0}2^{-jH}(1+j)^{1/\a+\de} \sum_{k\in\Z} \big |\Psi_{\a,H}(2^j t'-k)-\Psi_{\a,H}(2^j t''-k)\big |\\
& \le c_2 c_3 |t'-t''| (1+j_0)^{1/\a+\de}\sum_{j=0}^{j_0}  2^{j(1-H)}\\
&\le c_4 \big |t'-t'' \big |^{H}\big (\log(1+|t'-t''|^{-1})\big)^{1/\a+\de},  
\end{split}
\end{equation}
where the positive finite constant $c_4$ does not depend on $j_0$, $t'$ and $t''$. On the other hand, one can derive from the triangle 
inequality, (\ref{thm:main1:eq8}) and (\ref{thm:main1:eq8bis}) that, for every $j\ge j_0+1$,
\[
\sum_{k\in\Z} \big |\Psi_{\a,H}(2^j t'-k)-\Psi_{\a,H}(2^j t''-k)\big |\le 2 c_2 c_3,
\]
and consequently 
\begin{eqnarray}
\label{thm:main1:eq13}
&& \sum_{j=j_0+1}^{+\infty}2^{-jH}(1+j)^{1/\a+\de} \sum_{k\in\Z} \big |\Psi_{\a,H}(2^j t'-k)-\Psi_{\a,H}(2^j t''-k)\big |\nonumber\\
&& \le 2 c_2 c_3 2^{-(j_0+1)H} \sum_{p=0}^{+\infty}2^{-pH}(2+j_0+p)^{1/\a+\de} \nonumber\\
&& \le 2 c_2 c_3 2^{-(j_0+1)H}(2+j_0)^{1/\a+\de}  \sum_{p=0}^{+\infty}2^{-pH}\Big (1+\frac{p}{2+j_0}\Big)^{1/\a+\de} \nonumber\\
&& \le \Big (2 c_2 c_3  \sum_{p=0}^{+\infty}2^{-pH}(1+p)^{1/\a+\de}\Big) 2^{-(j_0+1)H}(2+j_0)^{1/\a+\de} \nonumber\\
&& \le c_5 \big |t'-t'' \big |^{H}\big (\log(1+|t'-t''|^{-1})\big)^{1/\a+\de},
\end{eqnarray}
where the last inequality follows from (\ref{thm:main1:eq9}) and (\ref{thm:main1:eq10}) and where $c_5$ is a positive and
finite constant not depending on $j_0$, $t'$ and $t''$. Putting together (\ref{thm:main1:eq7}), (\ref{thm:main1:eq12}) and
 (\ref{thm:main1:eq13}) yields (\ref{thm:main1:eq5}).

Next we show that (\ref{thm:main1:eq6}) is satisfied. Let $\O^{**}$ be the event of probability $1$ on which (\ref{ineq2f:ep}) 
holds, and let $t'$ and $t''$ be two arbitrary real numbers such that $-\wt{\rho}\le t'<t''\le \wt{\rho}$. It follows from 
(\ref{ineq2f:ep}) and (\ref{thm:main1:eq4}) that on $\O^{**}$,
\begin{eqnarray}
\label{thm:main1:eq14}
&& \big | X_2^+(t')-X_2^+(t'')\big | \le \sum_{j=0}^{+\infty}\sum_{k\in\uK} 2^{-jH}\big |\Re\big (\ep_{\a,j,k}\big)\big |\big |\Psi_{\a,H}(2^j t'-k)-\Psi_{\a,H}(2^j t''-k)\big |\nonumber\\
&&\le  C_6 \sum_{j=0}^{+\infty}2^{-jH}(1+j)^{1/\a+\de} \sum_{k\in\uK} \sqrt{\log \big (3+j+|k|\big)}\,\big |\Psi_{\a,H}(2^j t'-k)-\Psi_{\a,H}(2^j t''-k)\big |\nonumber\\
&&\le  C_7 \sum_{j=0}^{+\infty}2^{-jH}(1+j)^{1/\a+2\de} \sum_{k\in\uK} \sqrt{\log \big (3+|k|\big)}\,\big |\Psi_{\a,H}(2^j t'-k)-\Psi_{\a,H}(2^j t''-k)\big |,\nonumber\\
\end{eqnarray}
where $C_6$ and $C_7$ are two positive finite random variables not depending on $t'$ and $t''$. By the mean-value 
Theorem and (\ref{thm:main1:eq8}), we see that for all $j\in\Z_+$ and $k\in\uK$,
\begin{equation}
\label{thm:main1:eq15}
\begin{split}
\big |\Psi_{\a,H}(2^j t'-k)-\Psi_{\a,H}(2^j t''-k)\big |&= 2^j |t'-t''| \big |\Psi_{\a,H}'(a-k)\big |\\
&\le  c_2 2^j |t'-t''| \big (1+|a-k|\big)^{-3} \\
&\le  c_2 2^j |t'-t''| \big (1+||k|-|a||\big)^{-3}, 
\end{split}
\end{equation}
where $a$ is some real number satisfying $2^j t'< a <2^j t''$ which implies that
\begin{equation}
\label{thm:main1:eq16}
|a|\le 2^j\wt{\rho}.
\end{equation}
Combining (\ref{thm:main1:eq15}) and (\ref{thm:main1:eq16}) with (\ref{thm:main1:eq0bis}),  we obtain that for all $j\in\Z_+$ and $k\in\uK$,  
\begin{equation}
\label{thm:main1:eq17}
\big |\Psi_{\a,H}(2^j t'-k)-\Psi_{\a,H}(2^j t''-k)\big |\le c_2 2^j |t'-t''| \big (1+|k|- 2^j\wt{\rho}\,\big)^{-3}
\end{equation}
Then (\ref{thm:main1:eq0bis}) and (\ref{thm:main1:eq17}) entail, for every $j\in\Z_+$, that
\begin{eqnarray}
\label{thm:main1:eq18}
&& \sum_{k\in\uK} \sqrt{\log \big (3+|k|\big)}\,\big |\Psi_{\a,H}(2^j t'-k)-\Psi_{\a,H}(2^j t''-k)\big |\nonumber\\
&&\le c_2 2^{j+1} |t'-t''| \sum_{k=[2^j (\wt{\rho}+1)]+1}^{+\infty}\sqrt{\log \big (3+k\big)}\big (1+k- 2^j\wt{\rho}\,\big)^{-3}\nonumber\\
&&\le c_2 2^{j+1} |t'-t''|  \sum_{q=0}^{+\infty} \sqrt{\log \big (4+q+2^j (\wt{\rho}+1)\big)}\big (1+q+2^j\big)^{-3}\nonumber\\
&&\le c_2 2^{j+1} |t'-t''| \sqrt{\log \big (3+2^j \wt{\rho}\,\big)}\,\sum_{q=0}^{+\infty} \sqrt{\log \big (4+q+2^j \big)}\,\big (1+q+2^j\big)^{-3}\nonumber\\
&&\le c_8 |t'-t''| \, 2^{j}\sqrt{j+1} \,\sum_{q=0}^{+\infty} (1+q+2^j\big)^{-5/2}\nonumber\\
&&\le c_8 |t'-t''| \, 2^{j}\sqrt{j+1} \,\int_{0}^{+\infty} (x+2^j\big)^{-5/2}\,dx\nonumber\\
&& \le c_8 |t'-t''| \, 2^{-j/2}\sqrt{j+1}, 
\end{eqnarray}
where $c_8$ is a positive finite constant not depending on $j$, $t'$ and $t''$. Next,  it follows from (\ref{thm:main1:eq14}) and 
(\ref{thm:main1:eq18}) that on the event $\O^{**}$ of probability $1$,
\begin{equation}
\label{thm:main1:eq19}
\big | X_2^+(t')-X_2^+(t'')\big |\le C_9  |t'-t''|,
\end{equation}
where the positive finite random variable 
\[
C_9:=c_8 C_7\Big (\sum_{j=0}^{+\infty}2^{-j(H+1/2)}(1+j)^{1/2+1/\a+2\de}\Big)
\]
does not depend on $t'$ and $t''$. Finally, (\ref{thm:main1:eq19}) implies that  (\ref{thm:main1:eq6}) holds.
\cqfd
\bigskip

\noindent {\bf Proof of Theorem \ref{thm:main2}} \  Let $u<v$ be any fixed real numbers. 
For each $j\in\Z_+$, set
\[
\kj:=\big [ 2^{j-1}(u+v)\big].
\]
Then  
\begin{equation}
\label{main2:e4}
\big | 2^{-j}\,\kj-2^{-1}(u+v)\big |<2^{-j}.
\end{equation}
Let $\th$ be an even real-valued function in the Schwartz class $S(\R)$ whose Fourier transform $\wh{\th}$ (which is also an even 
real-valued function) has a compact support such that 
\begin{equation}
\label{main2:e1}
\supp\,\wh{\th}\subseteq \big\{\xi\in\R,\,\,2^{-1}\le |\xi|\le 1\big\}.
\end{equation}
For all $j\in\Z_+$,  let
\begin{equation}
\label{main2:e2}
W_{j}:=2^{j}\int_{\R} \th (2^j t-\kj) X(t)\,dt=\int_{\R}\th (t)\big (X(2^{-j}\,\kj+2^{-j} t)-X(2^{-j}\,\kj)\big)\,dt.
\end{equation}
Notice that the second equality in (\ref{main2:e2}) follows from the change of variable $t'=2^j t-\kj$ and the equality 
$\int_{\R}\th (t)\,dt=\wh{\th}(0)=0$ (see (\ref{main2:e1})). It is known from Proposition 5.1.4 and Remark 5.1.5 in \cite{Boutard} 
that the pathwise Lebesgue integrals in (\ref{main2:e2}) are well-defined and almost surely
\begin{equation}
\label{main2:e3}
W_{j}=\Re\bigg(\int_{\R}\frac{e^{i 2^{-j}\kj\xi}\,\wh{\th}(2^{-j}\xi)}{|\xi|^{H+1/\a}}\, d\wt{M}_\a(\xi)\bigg).
\end{equation}
Observe that (\ref{main2:e1}) and (\ref{main2:e3}) imply that $(W_j)_{j\in\Z_+}$ is a sequence of independent real-valued 
S$\a$S random variables whose scale parameters satisfy, for every $j\in\Z_+$, 
\begin{equation}
\label{main2:e5}
\si (W_j)=c_1 2^{-jH},
\end{equation}
where the positive finite constant $c_1:=\big (\int_{\R} |\eta |^{-\a H-1} \big |\wh{\th}(\eta)\big |^\a\,d\eta\big)^{1/\a}$. 
Let us now show that 
\begin{equation}
\label{main2:e6}
\limsup_{j\rightarrow +\infty}\, 2^{jH} (j+1)^{-1/\a}\, |W_j |=+\infty\quad\mbox{(almost surely).}
\end{equation}
Recall from Chapter 1 of \cite{ST94} that there are two constants $0<c_2<c_3<+\infty$ such that any arbitrary real-valued  
S$\a$S random variable $Z$ with scale parameter $1$ satisfies
\begin{equation}
\label{main2:e7}
c_2 z^{-\a}\le\P\big (|Z|>z)\le c_3 z^{-\a} ,\quad\mbox{for all $z\in [1,+\infty)$.}
\end{equation}
By using the first inequality in (\ref{main2:e7}), (\ref{main2:e5}) and the fact that 
\[
\sum_{j=1}^{+\infty} \frac{1}{(j+1)\log(j+1)}=+\infty ,
\]
we derive that 
\begin{equation}
\label{main2:e7b}
\sum_{j=1}^{+\infty} \P \Big (\big (c_1 2^{-jH}\big)^{-1} (j+1)^{-1/\a} |W_j |>\big (\log (j+1)\big)^{1/\a}\Big)=+\infty.
\end{equation}
Since the random variables $(W_j)_{j\in\Z_+}$ are independent,  (\ref{main2:e6}) follows from \eqref{main2:e7b} and from 
the second part of the Borel-Cantelli Lemma.

Recall from Corollary 4.2 in \cite{AB17} (see also \cite{Boutard}) and the continuity property of paths of 
$\{X(t), t\in\R\}$ that, for any fixed arbitrarily small $\de>0$, there is a positive finite random variable $C_{4,\de}$ 
such that almost surely
\begin{equation}
\label{main2:e9}
\big | X(t)\big|\le C_{4,\de} \big (1+|t|^H\big)\log^{1/\a+\de}\big (3+|t|\big),\quad\mbox{for all $t\in\R$.}
\end{equation}
It follows from (\ref{main2:e4}) and (\ref{main2:e9}) that almost surely for all $j\in\Z_+$, 
\begin{eqnarray*}
&& \int_{\{|t|>2^{j/2}\}}\big |\th (t)\big| \big |X(2^{-j}\,\kj+2^{-j} t)\big |\,dt\\
&& \le C_{4,\de}  \int_{\{|t|>2^{j/2}\}}\big |\th (t)\big|\big (1+|2^{-j}\,\kj+2^{-j} t|^H\big)\log^{1/\a+\de}\big (3+|2^{-j}\,\kj+2^{-j} t|\big)\,dt\\
&& \le C_{4,\de}  \int_{\{|t|>2^{j/2}\}}\big |\th (t)\big|\big (2+(|u|+|v|+|t|)^H\big)\log^{1/\a+\de}\big (4+|u|+|v|+|t|\big)\,dt.
\end{eqnarray*}
Since $\th\in S(\R)$, we have
\begin{equation}
\label{main2:e10}
\lim_{j\rightarrow +\infty}\,2^{jH} (j+1)^{-1/\a}\int_{\{|t|>2^{j/2}\}}\big |\th (t)\big| \big |X(2^{-j}\,\kj+2^{-j} t)\big |\,dt=0 \quad\mbox{(almost surely).}
\end{equation}
Combining (\ref{main2:e2}) with (\ref{main2:e6}) and (\ref{main2:e10}) gives
\begin{equation}
\label{main2:e11}
\limsup_{j\rightarrow +\infty}\, 2^{jH} (j+1)^{-1/\a}\, |\wt{W}_j |=+\infty\quad\mbox{a.s.,}
\end{equation}
where 
\begin{equation}
\label{main2:e12}
\wt{W}_j:=\int_{\{|t|\le 2^{j/2}\}}\th (t)\big (X(2^{-j}\,\kj+2^{-j} t)-X(2^{-j}\,\kj)\big )\,dt.
\end{equation}
 
Let us now introduce the positive random variable $A$ defined as 
\begin{equation}
\label{main2:e12bis}
A:=\sup_{u\le t'<t''\le v}\,\frac{\big | X(t')-X(t'')\big |}{\big |t'-t'' \big |^{H}\big (\log(1+|t'-t''|^{-1})\big)^{1/\a}}.
\end{equation}
Observe that for proving the theorem, it is enough to show that 
\begin{equation}
\label{main2:e8}
A=+\infty \quad\mbox{a.s.}
\end{equation}
Let $j_0$ be a positive fixed integer which is large enough so that  
\begin{equation}
\label{main2:e13}
2^{-j/2}\le 2^{-2} (v-u),\quad\mbox{for all $j\ge j_0$.}
\end{equation}
By (\ref{main2:e4}) and (\ref{main2:e13}), we have  
\begin{equation}
\label{main2:e14}
2^{-j}\,\kj+2^{-j} t\in [u,v],\quad\mbox{for all $j\ge j_0$ and $t\in \big [-2^{j/2}, 2^{j/2}\big]$.}
\end{equation}
Then, it follows from (\ref{main2:e12}), (\ref{main2:e12bis}) and (\ref{main2:e14}) that for all $j\ge j_0$,  
\begin{eqnarray}
\label{main2:e15}
| \wt{W}_j | &\le & A \int_{\R} \big |\th (t)\big| \big |2^{-j} t\big |^{H}\big (\log(1+|2^{-j} t|^{-1})\big)^{1/\a}\,dt\nonumber\\
&\le & A 2^{-jH} \int_{\R} \big |\th (t)\big| |t |^{H}\big (\log(2^{j}+2^{j} | t|^{-1})\big)^{1/\a}\,dt\nonumber\\
&\le & A 2^{-jH}(j+1)^{1/\a} \int_{\R} \big |\th (t)\big| |t |^{H}\big (1+\log(1+| t|^{-1})\big)^{1/\a}\,dt.
\end{eqnarray}
Moreover, the fact that  $\th\in S(\R)$ implies 
\begin{equation}
\label{main2:e16}
\int_{\R} \big |\th (t)\big| |t |^{H}\big (1+\log(1+| t|^{-1})\big)^{1/\a}\,dt<+\infty.
\end{equation}
Finally, putting together (\ref{main2:e11}), (\ref{main2:e15}) and (\ref{main2:e16}) yields (\ref{main2:e8}). 
This finishes the proof of Theorem \ref{thm:main2}. 

\bigskip
\noindent{\bf Acknowledgments}:  The research of A. Ayache is partially supported by the Labex CEMPI (ANR-11-LABX-0007-01), the GDR 3475 (Analyse Multifractale et Autosimilarit\'e), and the Australian Research Council’s Discovery Projects funding scheme (project number DP220101680). The research of Y. Xiao is partially supported by NSF grant DMS-2153846. This work was partially written during A. Ayache’s visit to Michigan State University in June 2023; he is very grateful to this university for its hospitality and financial support.

\bibliographystyle{plain}
\begin{small}

\end{small}

\end{document}